\documentclass[12pt]{amsart}

\usepackage{amsmath,amssymb,setspace}
\usepackage[active]{srcltx}
\usepackage[pagebackref, colorlinks, linkcolor=red, citecolor=blue, urlcolor=blue, hypertexnames=true]{hyperref}
\usepackage[backrefs]{amsrefs}
\usepackage{enumerate}
\allowdisplaybreaks
\usepackage[all,cmtip]{xy}

\usepackage{color}

\newtheorem{thm}{Theorem}[section]
\newtheorem{cor}[thm]{Corollary}
\newtheorem{lem}[thm]{Lemma}

\newtheorem{definition}[thm]{Definition}
\newtheorem{example}[thm]{Example}
\newtheorem{remark}[thm]{Remark}
\newtheorem{proposition}[thm]{Proposition}

\theoremstyle{definition}


\begin{document}
\title[Chain and topoligical entropy for free semigroup actions]{Chain recurrence rates and topological entropy for free semigroup actions}
\author[Yanjie Tang, Xiaojiang Ye and Dongkui Ma]{}

\email{yjtang1994@gmail.com}
\email{yexiaojiang12@163.com}
\email{dkma@scut.edu.cn}

\date{\today}
\thanks{{$^{*}$}Corresponding author: Dongkui Ma}

\maketitle
\centerline{\scshape Yanjie Tang, Xiaojiang Ye and Dongkui Ma $^*$}
\medskip
{\footnotesize
	\centerline{School of Mathematics, South China University of Technology, }
	\centerline{Guangzhou 510641, P.R. China}
} 

\hspace{2mm}

\begin{abstract}
Misiurewicz \cite{MM} introduced the concept of pseudo-entropy and proved this quantity coincides with topological entropy.  Richeson et al. \cite{RD} obtained the lower bounded of topological entropy by means of the definition of pseudo-entropy. This paper aims to generalize the main results obtained by Misiurewicz and Richeson et al. to free semigroup actions. Firstly, the pseudo-entropy is introduced for free semigroup actions, and it is shown that the pseudo-entropy coincides with the topological entropy defined by Bufetov \cite{BAA}. Secondly, these concepts of the chain recurrence, the chain mixing, the chain recurrence time and the chain mixing time for free semigroup actions are introduced, and the upper bounds of these recurrence times are given. Forthermore, a lower bound of topological entropy is given by the lower box dimension and the chain mixing time using the definition of pseudo-entropy for free semigroup actions. Thirdly, the structure of chain transitive systems for free semigroup actions is discussed.
\end{abstract}


\section{Introduction}

\hspace{4mm}
Topological entropy was first introduced by Adler et al. \cite{AKK}. Bowen \cite{BR} and Dinaburg \cite{ADV} extended the topological entropy to a uniformly continuous map on metric space and proved that it coincides with that defined by Adler et al. for a compace metric space. The topological entropy turned out to be a surprisingly universal concept in dynamical systems since it appears in the study of different subjects such as fractal, Poincar\'{e} recurrence, and in the analysis of either local or global complexities. The dynamical systems for free semigroup actions is the natural extention of classical topological dynamical systems. Ghys et al. \cite{MR926526} proposed a definition of topological entropy for finitely generated pseudo-groups of continuous maps. Bi\'{s} \cite{BAC} and Bufetov \cite{BAA} respectively defined the topological entropy for free semigroup actions from different angles. Many remarkable results have been obtained \cite{JM,LM,CMR,MR3592853,MR3784991,MR926526}.

\hspace{4mm}
Pseudo-orbits, or chains, have always been  one of the significant tools for studying the properties of topological dynamical systems. In recent years, a large number of scholars have used pseudo-orbits or chains as tools to study topological entropy, and have obtained some excellent results \cite{MM,BM,RD,MR1336706,MR3539733}.  A remarkable result by Misiurewicz \cite{MM} stated that the topological entropy can be calculated by the exponential growth rate of the number of pseudo-orbits. Barge and Swanson \cite{BM} further found that relpace pseudo-orbits with periodic pseudo-orbits, and the result proved by Misiurewicz \cite{MM} was still valid. In \cite{MR1336706}, Hurley considered pseudo-orbits for inverse images and showed that the point entropy of pseudo-orbits is in fact equal to the topological entropy. Taking the topological entropy defined Misiurewicz \cite{MM} as a bridge, Richeson and Wiseman \cite{RD} related the chain mixing time and the lower box dimension to topological entropy, and obtained a lower bound of topological entropy. It is so interesing! But the above results for topological entropy's estimation focus on a single map. A natural question emerges here, whether we can give a lower bound of topological entropy for free semigroup actions.   To answer this question, we introduce the notions of pseudo-entropy for free semigroup actions, and the Sect. \ref{TCF} and Sect. \ref{entropy}  in this paper give an affirmative answer to this question.

\hspace{4mm}
Akin \cite{MR1219737} (Exercise 8.22 and 9.18) initially discussed the structure of the chain transitive maps, a map of chain transitive but not chain mixing factors a cyclic permutation on a finite set with at least two elements. Richeson and Wiseman \cite{RD} enriched the result of Akin \cite{MR1219737} and filled in the gaps in the proofs sketched. They obtained that if $f$ is a chain transitive map on a compact metric space either then there is a period $k\geq 1$ such that $f$ cyclically permutes $k$ closed and open equivalence classes and $f^k$ restricted to each equivalence class is chain mixing; or $f$ factors onto an adding machine map. The above results for structure of chain transitive systems focus on a single map. Naturally, we wonder if the result of Richeson and Wiseman \cite{RD} remains valid in case of free semigroup actions. 

\hspace{4mm}
Let $X$ be a compact metric space with metric $d$, and $G$ be the free semigroup acting on the space $X$ generated by $\{f_0,\cdots,f_{m-1}\}$. In the following theorem, denote by $h(G)$ and $h^*(G)$ the topological entropy and the pseudo-entropy for the free semigroup action $G$, respectively (see Sect. \ref{TCF}). Let $r_\varepsilon(G)$ and $m_\varepsilon(\delta,G)$ be the chain recurrence time and the chain mixing time of the free semigroup action $G$, more precisely in \ref{entropy}. 

\hspace{4mm}
Now we start to state our main theorem.
\begin{thm}
	\label{CE}
	The topogogical entropy $h(G)$ coincides with the pseudo-entropy $h^*(G)$ for free semigroup actions.
\end{thm}
\begin{thm}\label{UBD}
	Let $\bar{b}$ be the upper box dimension of $X$. There exists a constant $C>0$ such that for small enough $\varepsilon>0$:
	\begin{enumerate}[(1)]
		\item if $G$ is chain transitive, then $r_\varepsilon(G)\leq  C/{\varepsilon^{\bar{b}+1}}$;
		\item if $G$ is chain mixing, then $m_\varepsilon(\delta,G)\leq C/{\varepsilon^{2(\bar{b}+1)}}$.
	\end{enumerate}
\end{thm}

\begin{thm}
	Let $(X,G)$ be chain mixing. Then the topological entropy $h(G)$ satisfies,
	$$
	h(G)\geq \max\left \{0,\,\,\underline{b}\cdot\limsup_{\delta\to 0}\dfrac{\log (1/\delta)}{\lim_{\varepsilon\to 0 }m_\varepsilon(\delta)}-\log m\right \},
	$$
	where $\underline{b}$ is the lower box dimension of $X$.
	\label{LBM}
\end{thm}
\begin{thm}
	\label{EC}
	Let $G$ be chain transitive. Then either
	\begin{enumerate}[(1)]
		\item There is a period $k\geq 1$, such that $G$ cyclically permutes $k$ closed and open equivalence classes of $X$, and $G^k$ restricted to each equivalence classes is chain mixing; or
		\item $G$ factors onto an adding machine map.
	\end{enumerate}
\end{thm}

\hspace{4mm}
 This paper is organized as follows. In Sect.  \ref{PL}, we give some preliminaries. In Sect. \ref{TCF}, above all, the pseudo-separated set and the pseudo-spanning set for free semigroup actions are introduced; next, we naturally define the pseudo-entropy for free semigroup actions and prove Theorem \ref{CE}. In Sect. \ref{entropy}, we introduce these concepts of the chain recurrence, the chain mixing, the chain recurrence time and the chain mixing time for free semigroup actions, and prove Theorem \ref{UBD}. Forthermore, Theorem \ref{LBM} is obtained by means of the definition of the pseudo-entropy for free semigroup actions.  In Sect. \ref{chain}, we discuss the structure of chain transitive systems and prove Theorem \ref{EC}. Our analysis generalizes the results  obtained by Misiurewicz \cite{MM}, Bufetov \cite{BAA} and Richeson et al. \cite{RD}.

\section{Preliminaries}\label{PL}

\hspace{4mm}
Denote $\mathbb{N}_0$, $\mathbb{N}$, and $\mathbb{Z}$ as the sets of non-negative integers, positive integers and integers, respectively.

\hspace{4mm}
Let $(X,d)$ be a compact metric space and $f$ be a continuous map on $X$. A \emph{$\delta$-pseudo-orbit} is an infinite sequence $(x_i)_{i=0}^\infty$ such that $d(f(x_i),x_{i+1})\leq\delta$ for $i\geq 0$. We say that $f$ has the \emph{pseudo-orbit tracing property} if for $\varepsilon> 0$ there is $\delta_\varepsilon> 0$ such that each $\delta_\varepsilon$-pseudo-orbit can be $\varepsilon$-shadowed, that is, if $(x_i)_{i=0}^\infty$ is a $\delta$-pseudo-orbit, then there exists $z\in X$ such that $d(f^i(z),x_i)<\varepsilon$ for all $i\geq 0$.

\hspace{4mm}
We recall the definition of pseudo-entropy of $f$. The first one is due to Misiurewicz \cite{MM}. Say a collection $E$ of $\delta$-pseudo-orbits of $f$ is $(n,\varepsilon)$-separated if, for each $(x_i),(y_i)\in E$, $(x_i)\neq (y_i)$, there is a $k$, $0\leq k\leq n-1$, for which $d(x_k,y_k)\geq\varepsilon$. Denote by  $s(n,\varepsilon,\delta)$ the maximal cardinality of an $(n,\varepsilon)$-separated set of $\delta$-pseudo-orbits.

\hspace{4mm}
A collection $K$ of $\delta$-pseudo-orbits of $f$ is $(n,\varepsilon)$-spanning if for each $\delta$-pseudo-orbit $(x_i)$, there is $(y_i)\in K$ such that $d(x_i,y_i)<\varepsilon$ for all $0\leq i\leq n-1$. The minimum cardinality of a $(n,\varepsilon)$-spanning set of $\delta$-pseudo-orbits is denoted by $r(n,\varepsilon,\delta)$.

\hspace{4mm}
Let
$$
h^*(f,\varepsilon,\delta)=\limsup_{n\to\infty}\frac{1}{n}\log s(n,\varepsilon,\delta),
$$
$$
h^* (f,\varepsilon)=\lim_{\delta\to 0 }h^*(f,\varepsilon,\delta),
$$
and
$$
h^* (f)=\lim_{\varepsilon\to 0}h^* (f,\varepsilon).
$$

\hspace{4mm}
The number $h^* (f)$ is called the \emph{pseudo-entropy} of $f$.

\hspace{4mm}
Obviously,
$$
r(n,\frac{\varepsilon}{2},\delta)\geq s(n,\varepsilon,\delta)\geq r(n,\varepsilon,\delta).
$$
Thus,
$$
h^* (f)=\lim_{\varepsilon\to 0}\lim_{\delta\to 0 }\limsup_{n\to\infty}\frac{1}{n}\log r(n,\varepsilon,\delta).
$$

\begin{thm}
	\label{PE}
	\cite{MM}
	The topological entropy $h(f)$ coincides with the pseudo-entropy $h^*(f)$.
\end{thm}

\hspace{4mm}
Let $F_m^+$ be the set of all finite words of symbols $0,1,\cdots,m-1$.  For any $w\in F_m^+$, $|w|$ stands for the length of $w$, that is, the digits of symbols in $w$. Obviously, $F^+_m$ with respect to the law of composition is a free semigroup with $m$ generators. We write $w'\leq w$ if there exists a word $w''\in F^+_m$ such that $w=w''w'$. For $w=i_0i_1\cdots i_k\in F^+_m$, denote $\overline{w}=i_k\cdots i_1i_0$.

\hspace{4mm}
Denote by $\Sigma^+_m$ the set of all one-side infinite sequences of symbols $0,1,\cdots,m-1$, that is, 
$$
\Sigma^+_m=\left\{\omega=(i_0i_1\cdots)\,\big  | \,i_k=0,1,\cdots,m-1,\: k\in\mathbb{N}_0\right\}
$$

\hspace{4mm}
The metric on  $\Sigma^+_m$ by setting
\begin{center}
$d'(\omega,\omega')=\frac{1}{m^k}$, where $k=\inf\{n\,\big  | \,i_n\neq i'_n\}$.
\end{center}

\hspace{4mm}
Obviously, $\Sigma^+_m$ is compact with respect to this metric. The shift $\sigma:\Sigma^+_m\to \Sigma^+_m $ is given by the formula, for each $\omega=(i_0i_1\cdots)\in\Sigma^+_m$,
$$
\sigma(\omega)=(i_1i_2\cdots).
$$

\hspace{4mm}
Suppose that $\omega\in\Sigma^+_m$, and $a,b\in \mathbb{N}$ with $a\leq b$. We write $\omega|_{[a,b]}=w$ if $w=i_ai_{a+1}\cdots i_b$.

\hspace{4mm}
Let $G$ be a free semigroup generated by $m$ generators $f_0,f_1,\cdots,f_{m-1}$  which are continuous maps on $X$, denoted as $G:=\{f_0,f_1,\cdots,f_{m-1}\}$. To each $w\in F^+_m$, $w=i_0i_1\cdots i_{k-1}$, let $f_w=f_{i_0}f_{i_1}\cdots f_{i_{k-1}}$.  Obviousely, $f_{ww'}=f_wf_{w'}$. We assign a metric $d_w$ on $X$ by setting
$$d_w(x_1,x_2)=\max_{w'\leq \overline{w}}d\left (f_{w'}(x_1),f_{w'}(x_2)\right ).
$$

\hspace{4mm}
A subset $B$ of $X$ is called a \emph{$(w,\varepsilon,G)$-spanning} subset if, for any  $x\in X$, there exists $y\in B$ with $d_w(x,y)<\varepsilon$. The minimum cardinality of a $(w,\varepsilon,G)$-spanning subset of $X$ is denoted by $B(w,\varepsilon,G)$.

\hspace{4mm}
A subset $K$ of $X$ is called a \emph{$(w,\varepsilon,G)$-separated} subset if, for any $x_1,x_2\in K$ with $x_1\neq x_2$, one has $d_w(x_1,x_2)\geq \varepsilon$. The maximum cardinality of a $(w,\varepsilon,G)$-separated subset of $X$ is denoted by $N(w,\varepsilon,G)$.

\hspace{4mm}
Let
$$
B(n,\varepsilon,G)=\frac{1}{m^n}\sum_{|w|=n}B(w,\varepsilon,G),
$$

$$
N(n,\varepsilon,G)=\frac{1}{m^n}\sum_{|w|=n}N(w,\varepsilon,G).
$$

\hspace{4mm}
In \cite{BAA}, the author introduced the topological entropy for free semigroup actions.
    The topological entropy of free semigroup actions is defined by the formula
    $$
    \begin{aligned}
    	h(G)&=\lim_{\varepsilon\to 0}\limsup_{n\to\infty}\frac{1}{n}\log B(n,\varepsilon,G)\\
    	&=\lim_{\varepsilon\to 0}\limsup_{n\to\infty}\frac{1}{n}\log N(n,\varepsilon,G).
    \end{aligned}
    $$

\hspace{4mm}
The dynamical systems for free semigroup actions have a strong connection with skew-products which has been analyzed to obtain properties of free semigroup actions through fiber associated with the skew-product (see for instance \cite{MR4200965}). Recall that the skew-product transformation by given as follows:
$$
F:\Sigma^+ _m \times X\to\Sigma^+ _m \times X,\:\, (\omega,x)\mapsto \big(\sigma(\omega),f_{i_0}(x)\big),
$$
where $\omega=(i_0 i_1\cdots)$ and $\sigma$ is the shift map of $\Sigma^+ _m $. And the metric $D$ on $\Sigma^+ _m \times X$ be given by the formula
$$
D((\omega,x),(\omega',x'))=\max\{d'(\omega,\omega'),d(x,x')\}.
$$

\begin{thm}
	\label{sp}
	\cite{BAA}Topological entropy of the skew-product transformation $F$ satisfies
	$$h(F)=\log m +h(G).$$
\end{thm}

\hspace{4mm}
We recall the definitions of box dimension more precisely in \cite{FK}. Let $E$ be a non-empty subset of $X$. Let $N_\delta(E)$ be the smallest number of stes of diameter at most $\delta$ which can cover $F$. The \emph{lower} and \emph{upper box dimensions} of $E$ respectively are defined as
$$
\underline{\dim}_B E=\liminf_{\delta\to 0}\frac{\log N_\delta (E)}{-\log\delta},
$$
and
$$
\overline{\dim}_B E=\limsup_{\delta\to 0}\frac{\log N_\delta (E)}{-\log\delta}.
$$

\section{The pseudo-entropy for free semigroup actions}\label{TCF}
In this section we will introduce the concept of pseudo-entropy for free semigroup actions and prove Theorem \ref{CE}.

\hspace{4mm}
According to Bahabadi \cite{BZ}, recall that for $w=i_0\cdots i_{n-1}\in F^+_m$, a \emph{$(w,\delta)$-chain} (or $(w,\delta)$-pseudo-orbit) of $G$ from $x$ to $y$ is a sequence $(x_0=x,\cdots,x_n=y)$ such that $d(f_{i_j}(x_j),x_{j+1})\leq\delta$ for $j=0,\cdots,n-1$. To simplify notation, we usually write $(x_j)_{j=0}^n$. For $w\in F^+ _m$ with $|w|=n$, denote 
\begin{center}
$E(w,\delta):=\big\{(x_i)_{i=0} ^n \ |\ (x_i)_{i=0}^n$ is a $(w,\delta)$-chain$\big\}$.
\end{center}

\hspace{4mm}
 Similar to Misiurewicz \cite{MM}, we mimic this definition of Bufetov \cite{BAA} by  pseudo-orbits  to introduce these following definitions for free semigroup actions.
 A subset $B$ of $E(w,\delta)$ is called a \emph{$(w,\varepsilon,\delta,G)$-pseudo-spanning set} of $X$ if, for any $(x_i)_{i=0} ^n \in E(w,\delta)$, there is a $(y_i)_{i=0} ^n\in B$, such that $d(x_i,y_i)\leq\varepsilon$ for every $0\leq i<n$. The minimum cardinality of a $(w,\varepsilon,\delta,G)$-pseudo-spanning set of $X$ is denoted by  $B^*(w,\varepsilon,\delta,G)$.
 A subset $K$ of $E(w,\delta)$ is called a \emph{$(w,\varepsilon,\delta,G)$-pseudo-separated set} of $X$ if, for any $(x_i)_{i=0} ^n$, $(y_i)_{i=0} ^n \in  K$, $(x_i)_{i=0} ^n\neq (y_i)_{i=0} ^n$, there is some $0\leq i<n$, such that $d(x_i,y_i)>\varepsilon$. 
The maximum cardinality of a $(w,\varepsilon,\delta,G)$-pseudo-separated set of $X$ is denoted by  $N^*(w,\varepsilon,\delta,G)$.

\hspace{4mm}
Let
$$
B^*(n,\varepsilon,\delta,G)=\frac{1}{m^n}\sum_{|w|=n}B^*(w,\varepsilon,\delta,G),
$$

$$
N^*(n,\varepsilon,\delta,G)=\frac{1}{m^n}\sum_{|w|=n}N^*(w,\varepsilon,\delta,G).
$$

\hspace{4mm}
Obviously,
$$
B^*(w,\frac{\varepsilon}{2},\delta,G)\geq N^*(w,\varepsilon,\delta,G)\geq B^*(w,\varepsilon,\delta,G),
$$

whence,
$$
B^*(n,\frac{\varepsilon}{2},\delta,G)\geq N^*(n,\varepsilon,\delta,G)\geq B^*(n,\varepsilon,\delta,G).
$$

\begin{remark}
	If $G=\{f\}$, we use  $(n,\varepsilon,\delta,f)$-pseudo-separeted (spanning) set instead of $(n.\varepsilon,f)$-separeted (spanning) set of $\delta$-pseudo-orbits of $f$ as the two sets are the same.
\end{remark}

\hspace{4mm}
New let, 
$$ h^*(\varepsilon,\delta,G)=\limsup_{n\to\infty}\frac{1}{n}\log N^*(n,\varepsilon,\delta,G),
$$
$$
h^*(\varepsilon,G)=\lim_{\delta\to 0}h^*(\varepsilon,\delta,G),
$$
and 
$$
h^*(G)=\lim_{\varepsilon\to 0}h^*(\varepsilon,G).
$$
\begin{definition}
	The number $h^*(G)$ is called pseudo-entropy for the free semigroup action $G$.	
\end{definition}

\hspace{4mm}
It easily follows that
$$
h^*(G)=\lim_{\varepsilon\to 0}\lim_{\delta\to 0}\limsup_{n\to\infty}\frac{1}{n}\log B^*(n,\varepsilon,\delta,G).
$$ 
\begin{remark}
	If $G=\{f\}$, it is clear that $h^*(G)$ coincides with $h^*(f)$.
\end{remark}
\hspace{4mm}
Next, we will prove Theorem \ref{CE}. In fact, it is enough to show that $h^*(F)=\log m+h^*(G)$ by Theorem \ref{PE} and \ref{sp}.  To this end, we adopt the method of Bufetov \cite{BAA}. Hence, we need the following two lemmas.

\begin{lem}
\label{SE}
For any natural number $n\in \mathbb{N}$ and $0<\varepsilon,\delta<\frac{1}{2}$,
$$
N^*(n,\varepsilon,\delta,F)\geq\sum_{|w|=n}N^*(w,\varepsilon,\delta,G).
$$
\end{lem}

\begin{proof}
Let $N=m^n$, there are $N$ distinct words of length $n$ in $F^+ _m$. Denote these words by $w^{(1)},\cdots,w^{(N)}$. For any $i=1,\cdots,N$, let $\omega(i)\in\Sigma^+ _m$ be an arbitrary sequence such that 
$\omega(i)|_{[0,n-1]}=w^{(i)}$. Suppose that $B_i$ is a $(w^{(i)},\varepsilon,\delta,G)$-pseudo-separated set of maximum cardinality of $X$ for all $1\leq i\leq N$. For any $(x_k ^{(i)})_{k=0}^n\in B_i$, consider that
$$
\left (\big(\omega(i),x^{(i)}_0 \big),\big(\sigma(\omega(i)),x^{(i)}_1 \big),\cdots,\big(\sigma^{n-1}(\omega(i)),x^{(i)}_{n-1} \big),\big(\sigma^{n}(\omega(i)),x^{(i)}_{n} \big)\right ).
$$
It clear that it is an $(n,\delta)$-chain of $F$ as $(x_k ^{(i)})_{k=0}^n$ is a $(w^{(i)},\delta)$-chain of $G$.

\hspace{4mm}
Put
$$
K=\left \{\big(\sigma^{k}(\omega(i)),x^{(i)}_{k}\big)_{k=0}^n\,\left   | \, (x_k ^{(i)})_{k=0}^n\right .\in B_i, 1\leq i\leq N\right \}.
$$
We chaim that $K$ forms an $(n,\varepsilon,\delta,F)$-pseudo-separated set of $\Sigma^+ _m \times X$. Indeed, it suffices to check that these $(n,\delta)$-chains of $F$ determined by $B_i$ and $B_j$ are $\varepsilon$-separated where $i\neq j$ and $1\leq i,j\leq N$. For any $(x^{(i)}_k)_{k=0}^n\in B_i$ and $(x^{(j)}_k)_{k=0}^n\in B_j$, we have $(\sigma^{k}(\omega(i)),x^{(i)}_{k})_{k=0}^n, (\sigma^{k}(\omega(j)),x^{(j)}_{k})_{k=0}^n\in K$. Since $\omega(i)|_{[0,n-1]}=w^{(i)}$, $\omega(j)|_{[0,n-1]}=w^{(j)}$ and $w^{(i)}\neq w^{(j)}$, then $w^{(i)}_k\neq w^{(j)}_k$ for some $0\leq k\leq n-1$, this gives us  $d'(\sigma^k(\omega(i)),\sigma^k(\omega(j)))=1>\varepsilon$. Therefore $K$ is an $(n,\varepsilon,\delta,F)$-pseudo-separated set of $\Sigma^+ _m \times X$. The lemma is proved.
\end{proof}

\begin{lem}
\label{SP}
For any $\varepsilon>0$, there is some $\delta_\varepsilon>0$, for any $0<\delta<\delta_\varepsilon$ and  $n\in\mathbb{N}$, we have
$$
B^*(n,\varepsilon,\delta,F)\leq K(\varepsilon)\left(\sum_{|w|=n}B^*(w,\varepsilon,\delta,G)\right),
$$
where $K(\varepsilon)$ is a positive constant that depends only on $\varepsilon$.
\end{lem}

\begin{proof}
	
	For any $\varepsilon>0$,  let $C(\varepsilon)$ be a minimum positive integer such that $m^{-C(\varepsilon)}<\varepsilon$. Let $N=m^{n+C(\varepsilon)}$, there are $N$ distinct words of length $n+C(\varepsilon)$ in $F_m^+$. Denote these words by $w^{(1)},\cdots,w^{(N)}$.  For any $i=1,\cdots,N$, let $\omega(i)\in\Sigma^+ _m$ be an arbitrary sequence such that $\omega(i)|_{[0,n+C(\varepsilon)-1]}=w^{(i)}$. Since $(\Sigma_m^+,\sigma)$ has pseudo-orbit tracing property, there is $\delta_\varepsilon>0$, for any  $0<\delta<\delta_\varepsilon$, such that each $\delta$-pseudo-orbit of $\sigma$ can be $\varepsilon$-shadowed.
	Suppose that $E_i$ is  $(\omega(i)|_{[0,n-1]},\varepsilon,\delta,G)$-pseudo-spanning set of minimum cardinality of $X$ for all $i=1,\cdots,N$. For any $(y^{(i)}_k)_{k=0}^n$ in $E_i$ with $(\omega(i)|_{[0,n-1]},\delta)$-chain of $G$,  we can construte a $(n,\delta)$-chain of $F$  similar to Lemma \ref{SE},	that is, 
	$$
	\left  (\big (\omega(i),y^{(i)}_0\big ),\big(\sigma(\omega(i)),y^{(i)}_1\big),\cdots,\big(\sigma^{n-1}(\omega(i)),y^{(i)}_{n-1}\big),\big(\sigma^{n}(\omega(i)),y^{(i)}_{n}\big)\right ).
	$$
	
	\hspace{4mm}
	Put
	$$
	H=\left \{\big(\sigma^{k}(\omega(i)),y^{(i)}_k\big)_{k=0}^n \,\left  | \, (y^{(i)}_k)_{k=0}^n\in E_i,\:1\leq i\leq N\right .\right \}.
	$$
	We claim that $H$ forms an $(n,\varepsilon,\delta,F)$-pseudo-spanning set of $\Sigma_m^+\times X$. Indeed, suppose now that $(\omega^{(k)},x^{(k)})_{k=0}^n$ is a $(n,\delta)$-chain of $F$ where $\omega^{(k)}=(i^{(k)}_0 i^{(k)}_1\cdots)\in \Sigma_m^+$ for every $k=0, \cdots,n$. Clearly, $(\omega^{(0)},\omega^{(1)},\cdots,\omega^{(n)})$ is a $(n,\delta)$-chain of $\sigma$, by the pseudo-orbit tracing property of $(\Sigma^+_m,\sigma)$,  this implies that  there is an $\omega\in\Sigma^+_m$ such that $d'(\sigma^k(\omega),\omega^{(k)})<\varepsilon$ for all $0\leq k\leq n$. This yields that $\sigma^k(\omega)|_{[0,C(\varepsilon)-1]}=\omega^{(k)}|_{[0,C(\varepsilon)-1]}$. Moreover, we have $\omega|_{[0,n-1]}=i_0^{(0)}i_0^{(1)}\cdots i_0^{(n-1)}$. It is clear that $\omega|_{[0,C(\varepsilon)+n-1]}=\omega(i)|_{[0,C(\varepsilon)+n-1]}=w^{(i)}$ for some  $1\leq i\leq N$, this implies  $$\sigma^k(\omega)|_{[0,C(\varepsilon)-1]}=\sigma^k(\omega(i))|_{[0,C(\varepsilon)-1]},$$ and hence $\sigma^k(\omega(i))|_{[0,C(\varepsilon)-1]}=\omega^{(k)}|_{[0,C(\varepsilon)-1]}$ for all $0\leq k\leq n$.
	Since $(x^{(0)},x^{(1)},\cdots,x^{(n)})$ is an $(i_0^{(0)}i_0^{(1)}\cdots i_0^{(n-1)},\delta)$-chain of $G$ and 
	$$
	\omega(i)|_{[0,n-1]}=\omega|_{[0,n-1]}=i_0^{(0)}i_0^{(1)}\cdots i_0^{(n-1)},
	$$
	this gives us $(x^{(0)},x^{(1)},\cdots,x^{(n)})$ is an $(\omega(i)|_{[0,n-1]},\delta)$-chain of $G$. Therefore, there is $(y^{(i)}_0,y^{(i)}_1,\cdots,y^{(i)}_n)$ in $E_i$, such that $d(y^{(i)}_k ,x^{(k)})<\varepsilon$ for each $k=0,\cdots,n-1$. As $(y^{(i)}_0,y^{(i)}_1,\cdots,y^{(i)}_n)$ is an $(\omega(i)|_{[0,n-1]},\delta)$-chain of $G$,  we deduce that there exists an $(n,\delta)$-chain of $F$, that is,
	$$
	\left  (\big (\omega(i),y^{(i)}_0\big ),\big(\sigma(\omega(i)),y^{(i)}_1\big),\cdots,\big(\sigma^{n-1}(\omega(i)),y^{(i)}_{n-1}\big),\big(\sigma^{n}(\omega(i)),y^{(i)}_{n}\big)\right )\in H.
	$$
	such that 
	$$
	D\left (\big(\sigma^k(\omega(i)),y^{(i)}_k\big),\big(\omega^{(k)},x^{(k)}\big)\right )<\varepsilon
	$$
	for all $0\leq k\leq n.$ 
	Consequently, $H$ is an $(n,\varepsilon,\delta,F)$-pseudo-spanning set of $\Sigma_m^+\times X$. The number of $H$ is not greater than $K(\varepsilon)(\sum_{|w|=n}B^*(w,\varepsilon,\delta,G))$, where $K(\varepsilon)$ is a positive constant that depends only on $\varepsilon$. The lemma is proved.
\end{proof}

\hspace{4mm}
Now, we can obtain immediately Theorem \ref{CE}.
\begin{proof}[Proof of Theorem \ref{CE}] 
	From Lemma \ref{SE} we have
	$$N^*(n,\varepsilon,\delta,F)\geq\sum_{|w|=n}N^*(w,\varepsilon,\delta,G),$$
	whence, taking logarithms and limits, we obtain that
	$$h^*(F)\geq \log m+h^*(G).$$
	In this way, from Lemma \ref{SP}, we have
	$$B^*(n,\varepsilon,\delta,F)\leq K(\varepsilon)m^n B^*(n,\varepsilon,\delta,G),$$
	whence,
	$$h^*(F)\leq \log m+h^*(G).$$
	Combining these two inequalities we find that
	$$h^*(F)= \log m+h^*(G).$$
	We conclude by Theorem \ref{PE} and  \ref{sp} that
	$$h(G)=h^*(G).$$
	This finishes the proof of the theorem.
\end{proof}

\begin{remark}
	If $G=\{f\}$, then $h^*(f)$ coincides with the pseudo-entropy of $f$ defined by Misiurewicz \cite{MM}.
\end{remark}


\section{Chain recurrence rates and topological entropy}\label{entropy}
In this section, we mainly introduce these concepts of the chain recurrence and the chain mixing, the chain recurrence time and the chain mixing time for free semigroup actions and discuss some propositions of these notions, and prove Theorem \ref{UBD} and Theorem \ref{LBM}.

\hspace{4mm}
Let $(X,d)$ be a compact metric space and $G$ be a free semigroup generated by $m$ generators $f_0,f_1,\cdots,f_{m-1}$  which are continuous maps on $X$. We define the chain recurrence for free semigroup actions as follows:
\begin{definition}
	One says that  $x\in X$ is chain recurrence point of $G$ if for every $\varepsilon>0$, there is a $(w,\varepsilon)$-chain from $x$ to itself for some $w\in F^+_m$. We say that $G$ is chain recurrence if every point of $X$ is chain recurrent. 
\end{definition} 

\hspace{4mm}
Chain transitive and totally chain transitive of free  semigroup actions were introduced by \cite{BZ} and \cite{MR3842255}, respectively. We say $G$ is \emph{chain transitive} if for every $\varepsilon>0$ and any $x,y\in X$, there is an $(w,\varepsilon)$-chain from $x$ to $y$ for some $w\in F_m^+$. $G$ is \emph{totally chain transitive} if $G^k$ is chain transitive for all $k\geq 1$. According to \cite{RD}, we may define the chain mixing for free semigroup actions as follows:

\begin{definition}
	$G$ is said to be $\varepsilon$-chain mixing if there is an $N>0$ such that for any $x,y\in X$ and any $n\geq N$, there is a $(w,\varepsilon)$-chain from $x$ to $y$ for some $w\in F^+_m$ with $|w|=n$. $G$ is called chain mixing if it admits $\varepsilon$-chain mixing for every $\varepsilon>0$.
\end{definition}

\begin{remark}
	As compactness of $X$, a equivalent way to say $G$ is chain mixing is to say that for any $\varepsilon>0$ and $x,y\in X$, there is an $N>0$ such that for any $n\geq N$, there is a $(w,\varepsilon)$-chain from $x$ to $y$ for some $w\in F^+_m$ with $|w|=n$.
\end{remark}

\hspace{4mm}
If $x$ is a chain recurrence point, define the \emph{$\varepsilon$-chain recurrence time}  $r_\varepsilon(x,G)$ to be the smallest $n$ such that there is a $(w,\varepsilon)$-chain from $x$ to itself for some $w\in F^+_m$ with $|w|=n$. If $G$ is chain rucurrent, define $r_\varepsilon(G)$ to be the maximum over all $x$ of  $r_\varepsilon(x,G)$. To see that this maximum exists, observe that if there is a $(w,\varepsilon)$-chain from $x$ to itself for some $w\in F^+_m$, there is a neighboehood $U$ of $x$ such that for all $y\in U$, there is $(w,\varepsilon)$-chain from $y$ to itself   {for above $w$}. Then the compactness of $X$ gives an upper bound on $r_\varepsilon(G)$. 

\hspace{4mm} 
If $G$ is chain mixing, for $0<\varepsilon<\delta$ and $x\in X$, define \emph{chain mixing time} $m_\varepsilon(x,\delta, G)$ to be the smallest $N$ such that for any $n\geq N$ and $y\in X$, there exists a $(w,\varepsilon)$-chain from some point in $B(x,\delta)$ to $y$ for some $w\in F^+_m$ with $|w|=n$. We define $m_\varepsilon(\delta, G)$ to be the maximum over all $x$ of $m_\varepsilon(x,\delta, G)$. The muximum exists by compactness.

\hspace{4mm} 
Inspired by Example 2 in \cite{RD}, we give the following Example \ref{eg4.1} to illustrate the existence of a system for free semigroup actions that is chain transitive but not chain mixing.
\begin{example}
	\label{eg4.1}
	Let $X$ be the disjoint union of two circles and $G=\{f_0,f_1\}$ where $f_0$ is the map sending a point $x$ to the point $2x$  in the other circle and $f_1$ is the map sending a point $x$ to the point $3x$ in the other circle. Obviously, $G$ is chain transitive but not chain mixing, because it is not $\varepsilon$-chain mixing for any $\varepsilon$ smaller than the distance between the two circles. However, $G^2$ restricted  to one circle is chain mixing.
\end{example}

\hspace{4mm} 
Next we provide an example of chain mixing for free semigroup actions as follows.
\begin{example}
	We define two continuous maps $f_0$, $f_1$ on $\Sigma^+_2$ as follows:
	$$
	f_0(s_0s_1\cdots)=0s_0s_1\cdots,\;\;f_1(s_0s_1\cdots)=1s_0s_1\cdots.
	$$
	Put $G=\{f_0,f_1\}$. In \cite{BZ}, the author proved $G$ has the shadowing property.  We claim that $(\Sigma^+_2,G)$ is chain mixing. Indeed, suppose that $\varepsilon>0$ and $\omega'=(i_0i_1\cdots),\omega''=(j_0j_1\cdots)\in\Sigma^+_2$ are given. Note that there is an $N\in\mathbb{N}$ such that $\frac{1}{2^{N-1}}<\varepsilon$. Next we have to construct an $\varepsilon$-chain $(\omega_i)_{i=0}^n$ from $\omega'$ to $\omega''$ of length exactly $n$ with $n\geq N$. Firstly, put $\omega_0=\omega'$, and $\omega_n=\omega''$; Then, for $1\leq i\leq n-2$, let  $\omega_i=(j_{n-i}\cdots j_{n-2}j_{n-1}i_0i_{1}\cdots i_{N}\cdots)$ such that $\omega_i\big|_{[i,i+N]}=\omega'\big|_{[0,N]}=i_0\cdots i_N$. Then it is easy to see that $(\omega_i)_{i=0}^n$ is a $(w,\varepsilon)$-chain from $\omega'$ to $\omega''$ of length $n$ where $w=j_{n-1}j_{n-2}\cdots j_0$.
\end{example}

\hspace{4mm}
Let $(X,d)$ and $(Y,d_Y)$ be compact metric spaces. Let $G:=\{f_0,f_1,\cdots,f_{m-1}\}$ where $f_0,f_1\cdots,f_{m-1}$ are continuous maps on $X$, and $H:=\{g_0,g_1,\cdots,g_{n-1}\}$ where $g_0,g_1,\cdots,g_{n-1}$ are continuous maps on $Y$.
Let 
$$
G\times H:=\left \{(f\times g)_0,\cdots, (f\times g)_{mn-1}\right \},
$$ where $(f\times g)_i\in\{f_j\times g_k\:|\:0\leq j\leq m-1, 0\leq k\leq n-1\}$, and $(f\times g)(x,y)=(f(x),g(y))$ for any $f\times g\in G\times H$ and $x\in X$, $y\in Y$. A metric $d_{X\times Y}$ on the product space $X\times Y$ is given by 
$$
d_{X\times Y}\left ((x_1, y_1), (x_2, y_2)\right ) =\max \left \{{d(x_1, x_2), d_Y(y_1, y_2)}\right \}.
$$

\hspace{4mm}
For any $v=v_0\cdots v_{r-1}\in F^+_{mn}$, there exist unique $w^{(1)}=i_0\cdots i_{r-1}\in F^+_{m}$ and unique $w^{(2)}=j_0\cdots j_{r-1}\in F^+_{n}$ such that $(f\times g)_{v_l}=f_{i_l}\times g_{j_l}$ for any $0 \leq l \leq r-1$ and thus $(f\times g)_v=f_{w^{(1)}}\times g_{w^{(2)}}$. On the other hand, if $w^{(1)}=i_0\cdots i_{r-1}\in F^+_{m}$, $w^{(2)}=j_0\cdots j_{r-1}\in F^+_{n}$, there exists unique $v=v_0\cdots v_{r-1}\in F^+_{mn}$ such that $f_{i_l}\times g_{j_l}=(f\times g)_{v_l}$ for any $0 \leq l \leq r-1$ and thus $f_{w^{(1)}}\times g_{w^{(2)}}=(f\times g)_v$. Thus, the map $ v\to (w^{(1)}, w^{(2)})$ is an one-to-one correspondence.

\hspace{4mm}
For $k\in\mathbb{N}$,  we denote
$$
G^k:=\left \{f_w\;|\;w\in F_m^+,|w|=k\right \}:=\left \{(f)_0,(f)_1,\cdots,(f)_{m^k-1}\right \},	
$$
that is, $G^k$ denotes the free semigroup generated by $\{\;f_w \;| \;w\in F_m^+,|w|=k\}$. For any $u=u_0\cdots u_{r-1}\in F_{m^k}^+$, there exists unique $w^{(0)},\cdots, w^{(r-1)}\in F^+_m$ with $|w^{(i)}|=k$ for all $0\leq i\leq r-1$ such that $(f)_{u_i}=f_{w^{(i)}}$ for all $0\leq i\leq r-1$ and thus $(f)_u=f_{w^{(0)}\cdots w^{(r-1)}}$. On the other hand, if $w^{(0)},\cdots,w^{(r-1)}\in F_m^+$ with $|w^{(i)}|=k$ for all $0\leq i\leq r-1$, there exists unique $u=u_0\cdots u_{r-1}\in F_{m^k}^+$ such that $f_{w^{(i)}}=(f)_{u_i}$ for all $0\leq i\leq r-1$, this implies that $f_{w^{(0)}\cdots w^{(r-1)}}=(f)_u$. Consequently, the map $u\to w^{(0)}\cdots w^{(r-1)}$ is an one-to-one correspondence.

\begin{proposition}
	\label{5.1}
	$(X,G)$ and $(Y,H)$ are chain recurrence if and only if
	$(X\times Y,G\times H)$ is chain recurrence. 
	If $(X,G)$ and $(Y,H)$ both are chain recurrence. Then for all $\varepsilon>0$, $x\in X$, and $y\in Y$, 
	\begin{enumerate}[(1)]
		\item \label{5.11}$r_\varepsilon((x,y),G\times H)$$\geq \max\{r_\varepsilon(x,G), r_\varepsilon(y,H)\}$;
		\item \label{5.12}$r_\varepsilon((x,y),G\times H)$$\leq \mathop{lcm} (r_\varepsilon(x,G), r_\varepsilon(y,H))$,	
	\end{enumerate}
	where $\mathop{lcm}(\cdot)$ denotes the least common multiple.
\end{proposition}

\begin{proof}
	Observe that  if $((x_0,y_0),\cdots,(x_k,y_k))$ is a $(v,\varepsilon)$-chain with $v=v_0\cdots v_{k-1}\in F^+_{mn}$ for $G\times H$, we have $(x_0,\cdots,x_k)$ is a $(w^{(1)},\varepsilon)$-chain with  some $w^{(1)}\in F^+_{m}$ for $G$, and $(y_0,\cdots,y_k)$ is a  $(w^{(2)},\varepsilon)$-chain with some $w^{(2)}\in F^+_{n}$ for $H$. This shows statement (\ref{5.11}), and $(X,G)$ and $(Y,H)$ both are chain recurrence if $(X\times Y,G\times H)$ is.
	
	\hspace{4mm}
	If $(X,G)$ and $(Y,H)$ both are chain recurrence, let $(x_0=x,\cdots,x_k=x)$ be  a $(w^{(1)},\varepsilon)$-chain with  some $w^{(1)}\in F^+_{m}$ for $G$,  and $(y_0=y,\cdots,y_s=y)$ be a  $(w^{(2)},\varepsilon)$-chain with some $w^{(2)}\in F^+_{n}$ for $H$. Then the $(w^{(1)}\cdots w^{(1)},\varepsilon)$-chain	
	$$
	{(x_0,\cdots,x_k=x_0,x_1,\cdots,x_k=x_0,\cdots,x_k=x_0)}
	$$
	formed by concatenating $(x_0=x,\cdots,x_k=x)$ with itself $\frac{s}{gcd(k,s)}$ times has length $lcm(k,s)$, where $gcd(\cdot)$ denotes the greatest common divisor. As does the $(w^{(2)}\cdots w^{(2)},\varepsilon)$-chain 
	$$
	\left (y_0,\cdots,y_s=y_0,y_1\cdots,y_s=y_0,\cdots,y_s=y_0\right )
	$$
	formed by concatenating $(y_0=y,\cdots,y_s=y)$ with itself $\frac{k}{gcd(k,s)}$ times has length $lcm(k,s)$. Combining the two gives a $(v,\varepsilon)$-chain from $(x,y)$ to itself of $G\times H$ for some $v\in F^+_{mn}$. This shows (\ref{5.12}) and $(X\times Y,G\times H)$ is chain recurrence.
\end{proof}
\begin{proposition}\label{5.2}
	Let $k\in\mathbb{N}$. 	$(X,G)$ is chain recurrence if and only if $(X,G^k)$ is chain recurrence. 
	If $(X,G)$ is chain recurrence. Then for all $\varepsilon>0$, $x\in X$, 	
	\begin{enumerate}[(1)]
		\item \label{(3)}$r_\varepsilon(x,G^k)\geq \frac{1}{k} r_\varepsilon(x,G)$;
		\item \label{(4)}there exists an $\varepsilon'\leq\varepsilon$ such that $r_\varepsilon(x,G^k)\leq  r_{\varepsilon'}(x,G)$.
	\end{enumerate}
\end{proposition}
\begin{proof}
	Suppose that  $(x_0,\cdots,x_l)$ is an $(u,\varepsilon)$-chain for $G^k$ with $u=u_0\cdots u_{l-1}\in F^+_{m^k}$. There are $w^{(0)},\cdots,w^{(l-1)}\in F_m^+$ such that $(f)_{u_j}=f_{w^{(j)}}$ for all $j=0,\cdots,l-1$. Let $w^{(j)}=i_{k-1}^{(j)}\cdots i_{0}^{(j)}\in F^+_m$ for $j=0,\cdots,l-1$. Then we insert the orbit $(f_{i_{0}^{(j)}}(x_j),\cdots,f_{i_{k-2}^{(j)}\cdots i_{0}^{(j)}}(x_j))$ between $x_j$ and $x_{j+1}$, for all $j=0,\cdots,l-1$. That is,
	$$
	\left (x_0,f_{i_{0}^{(0)}}(x_0),\cdots,f_{i_{k-2}^{(0)}\cdots i_{0}^{(0)}}(x_0), x_1,f_{i_{0}^{(1)}}(x_1),\cdots,f_{i_{k-2}^{(l-1)}\cdots i_{0}^{(l-1)}}(x_{l-1}),x_{l}\right ).
	$$
	It is clear that this is a $(w,\varepsilon)$-chain of length $lk$ for $G$, where $w= \overline{ w^{(0)}}\cdots \overline{w^{(l-1)}} \in F_m^+$. This last fact shows statement (\ref{(3)}), and proves that $G$ is chain recurrence if $G^k$ is.

    \hspace{4mm}
    Next, we show prove that $G^k$ is chain recurrence if $G$ is. By uniform continuity of $f_0,\cdots,f_{m-1}$, there is an $\varepsilon '<\varepsilon/k$  such that any $(w,\varepsilon')$-chain of length $|w|=k$ for $G$, $(x_0,\cdots,x_k)$, we have $d(f_{\overline{w}} (x_0),x_k)<\varepsilon$. For $x\in X$, we suppose that $(x_0=x,\cdots,x_{lk}=x)$ is a $(w',\varepsilon')$-chain of length $|w'|=lk$ of $G$ where $w'=i'_0\cdots i'_{lk-1}\in F^+_m$. We deduce that 
    $$
    d\left (f_{\overline{i'_{jk}\cdots i'_{{(j+1)}k-1}}}(x_{jk}),x_{{(j+1)}k}\right )<\varepsilon
    $$
    for all $0\leq j\leq l-1$. Moerover, there is $u=u_0\cdots u_{l-1}\in F^+_{m^k}$ such that $(f)_{u_j}=f_{\overline{i'_{jk}\cdots i'_{{(j+1)}k-1}}}$ for all $0\leq j\leq l-1$. Consequently, $(x_0,x_k,x_{2k},\cdots,x_{lk})$ is a $(u,\varepsilon)$-chain of length $|u|=l$ of $G^k$. 
    
    \hspace{4mm}
    Analogously, observe that by taking a $(w',\varepsilon')$-chain of length $|w'|=l$ for $G$ from $x$ to itself, and concatenating it with itself $k$ times, we can get a $(v,\varepsilon)$-chain of length $|v|=l$ for $G^k$ from $x$ to $x$ where $v\in F^+_{m^k}$. This shows (\ref{(4)}).			
\end{proof}

\begin{proposition}
	\label{4.8}
	Let $(X,G)$ and $(Y,H)$ be chain mixing, and $k\in\mathbb{N}$. Then $(X,G^k)$ and $(X\times Y,G\times H)$ are also chain mixing, and for all $\varepsilon>0$, 
	\begin{enumerate}[(1)]
		\item\label{(2.1)} $m_\varepsilon(\delta,G\times H)$$= \max\{m_\varepsilon(\delta,G), m_\varepsilon(\delta,H)\}$;
		\item\label{(2.2)} $m_\varepsilon(\delta,G^k)\geq \frac{1}{k} m_\varepsilon(\delta,G)$;
		\item \label{(2.3)}there exists an $\varepsilon'\leq\varepsilon$ such that $m_\varepsilon(\delta,G^k)\leq  m_{\varepsilon'}(\delta,G)$,	
	\end{enumerate}
\end{proposition}

\begin{proof}
	The fact that $G\times H$ is chain mixing if and only if both $G$ and $H$ are follow from the definition of chain mixing, thus we have statement (\ref{(2.1)}). The proof of (\ref{(2.2)}), (\ref{(2.3)}) and the fact that $G^k$ is chain mixing are analogous to those of the corresponding statements for chain recurrence in Proposition \ref{5.2}.
\end{proof}
\begin{remark}
	If $G=\{f\}$, then Proposition 24 and 26 of \cite{RD} are obtained by Proposition \ref{5.1}, \ref{5.2} and \ref{4.8}.
\end{remark}

\begin{proof}[Proof of Theorem \ref{UBD}]
	Let $s$ be the upper box dimension of $\Sigma_m^+\times X$, by Theorem 7.5 of \cite{FK}, 
	$$\overline{\dim}_B(\Sigma_m^+\times X)\leq \overline{\dim}_B(\Sigma_m^+)+\overline{\dim}_B(X).$$
	Since $\overline{\dim}_B(\Sigma_m^+)=1$, we have that $s\leq \bar{b}+1$.
	It follows from \cite{WX} that $F$ is chain transitive if and only if $G$ is chain transitive. 
	Note that if $\big((\omega_0,x_0),\cdots,(\omega_n,x_n)\big)$
	is a $(n,\varepsilon)$-chain of $F$, then
	$(x_0,\cdots,x_n)$ is a $(w,\varepsilon)$-chain of $G$ for some $w\in F^+_m$ with $|w|=n$. Thus, we have 
	$$r_\varepsilon(x,G)\leq\max_{\omega\in\Sigma_m^+}r_\varepsilon\left ((\omega,x),F\right ).$$
	This gives us $r_\varepsilon(G)\leq r_\varepsilon (F)$. By \cite{RD}, there is a $C>0$ such that $r_\varepsilon (F)\leq C/{\varepsilon^{s}}$, we deduce that 
	$$
	r_\varepsilon(G)\leq  C/{\varepsilon^{\bar{b}+1}}.
	$$ This shows (1).
	
	\hspace{4mm}
	Let $(\omega,x)\in \Sigma^+_m\times X$ and $m_\varepsilon((\omega,x),\delta,F)=N$. Since $F$ is chain mixing if and only if $G$ is chain mixing by \cite{WX}, this shows that  
	for any $(\xi,y)\in\Sigma^+_m\times X $,  there exist $(\zeta,x')\in B((\omega,x),\delta)$ and $\varepsilon$-chain of $F$ of length exactly $N$ from $(\zeta,x')$ to $(\xi,y)$, denoted by
	$$\big ((\zeta,x'), (\zeta_1,x_1),\cdots,(\zeta_N,x_N)=(\xi,y)\big ).
	$$
	It follows that $(x',x_1,\cdots,x_N=y)$ is a $(w,\varepsilon)$-chain for some $w\in F^+_m$ with $|w|=N$. We deduce that 
	$$m_\varepsilon(x,\delta,G)\leq\max_{\omega\in\Sigma_m^+}m_\varepsilon\left ((\omega,x),\delta,F\right ).$$
	This implies that $m_\varepsilon(\delta,G)\leq m_\varepsilon(\delta,F)$. By \cite{RD}, there is a $C>0$ such that $m_\varepsilon (\delta,F)\leq C/{\varepsilon^{2s}}$, we conclude that
	$$m_\varepsilon(\delta,G)\leq C/{\varepsilon^{2(\bar{b}+1)}}.$$ This shows (2).
\end{proof}

\begin{proof}[Proof of Theorem \ref{LBM}]
	By Theorem \ref{CE} we know,
	$$
	h(G)=\lim_{\varepsilon\to 0}\lim_{\delta\to 0}\limsup_{n\to\infty}\frac{1}{n}\log N^*(n,\varepsilon,\delta,G).
	$$  
	
	\hspace{4mm}
	For $\alpha>0$, let $N_{\alpha}:=N^*(0,\alpha,0,G)$. Let $\{x_1,\cdots,x_{N_{3\delta}}\}$ be a $3\delta$-separated set of points. Each sequence $(i_0,\cdots,i_k)\in\{1,\cdots,N_{3\delta}\}^{k+1}$ corresponds to $\{x_{i_0},\cdots,x_{i_k}\}$. Taking $x'_{i_k}=x_{i_k}$, by the definition of $m_\varepsilon(\delta)$, there are $x'_{i_{k-1}}\in B(x_{i_{k-1}},\delta)$ and $(w^{(k-1)},\varepsilon)$-chain of length $m_\varepsilon(\delta)$ for some $w^{(k-1)}\in F^+_m$ from $x'_{i_{k-1}}$ to $x'_{i_k}$. Now, suppose that this process continues until,  
	there are $x_{i_0}'\in B(x_{i_0},\delta)$ and $(w^{(0)},\varepsilon)$-chain of length $m_\varepsilon(\delta)$ for some $w^{(0)}\in F^+_m$ from $x'_{i_0}$ to $x'_{i_1}$. We can derive that $(x'_{i_0},\cdots,x_{x_k}')$ is a $(w^{(0)}\cdots w^{(k-1)},\varepsilon)$-chain of length $km_\varepsilon(\delta)$. Since the set $\{x_1,\cdots,x_{N_{3\delta}}\}$ is $3\delta$-separated, the set
	$$
	E:=\left \{ (x'_{i_1},\cdots,x'_{i_k})\:\big|\:(i_0,\cdots,i_k)\in\{1,\cdots,N_{3\delta}\}^{k+1} \right \}
	$$
	is a set of $\delta$-separated $\varepsilon$-chains of length $km_\varepsilon (\delta)$. Obviously, $\sharp E=(N_{3\delta})^{k+1}$, then,
	$$
	N^*(km_\varepsilon (\delta),\varepsilon,\delta,G)\geq \dfrac{(N_{3\delta})^{k+1}}{m^{km_\varepsilon (\delta)}}.
	$$
	
	\hspace{4mm}
	Next, let $B_{\alpha}:=B^*(0,\alpha,0,G)$ be the minimum number of balls of radius $\alpha$ necessary to cover $X$. By the definition of lower box dimension, for small enough $\alpha$, $B_{\alpha}\geq C(1/\alpha)^{\underline{b}}$ for some positive constant $C$. Clearly $N_{\alpha}\geq  B_{\alpha}$, then $N_{\alpha}\geq C(1/\alpha)^{\underline{b}}$ as well.
	
	\hspace{4mm}
	Finally, we have that
	$$
	\begin{aligned}
		h(G)&=\lim_{\delta\to 0}\lim_{\varepsilon\to 0}\limsup_{n\to\infty}\frac{1}{n}\log N^*(n,\delta,\varepsilon,G)\\
		&\geq\limsup_{\delta\to 0}\lim_{\varepsilon\to 0}\limsup_{k\to\infty}\frac{1}{km_\varepsilon (\delta)}\log N^*(km_\varepsilon (\delta),\delta,\varepsilon,G)\\
		&\geq\limsup_{\delta\to 0}\lim_{\varepsilon\to 0}\limsup_{k\to\infty}\frac{1}{km_\varepsilon (\delta)}\log{\frac{(N_{3\delta})^{k+1}}{m^{km_\varepsilon (\delta)}}}\\
		&=\limsup_{\delta\to 0}\lim_{\varepsilon\to 0}\frac{\log N_{3\delta}}{m_\varepsilon (\delta)}-\log m\\
		&\geq \limsup_{\delta\to 0}\lim_{\varepsilon\to 0}\frac{\log C(1/3\delta)^{\underline{b}}}{m_\varepsilon (\delta)}-\log m\\
		&= {\underline{b}}\cdot\limsup_{\delta\to 0}\dfrac{\log (1/\delta)}{\lim_{\varepsilon\to 0 }m_\varepsilon(\delta)}-\log m.
	\end{aligned}
	$$
	
	\hspace{4mm}
	Since $X$ is not a single point, and it cannot be a finite collection of points as $(X,G)$ is chain mixing, then       $\lim_{\delta\to 0}m_\varepsilon(\delta)=\infty$. Thus this conclusion is proved. 
\end{proof}
\begin{remark}
	In fact, by virtue of Theorem \ref{sp} and Theorem 28 of \cite{RD}, we immediately deduce that
	$$\begin{aligned}
		h(G)\geq \max\left \{0,\;{\underline{b}}\cdot\limsup_{\delta\to 0}\dfrac{\log (1/\delta)}{\lim_{\varepsilon\to 0 }m_\varepsilon(\delta, F)}-\log m\right \}.
	\end{aligned}$$

\hspace{4mm}
As $m_\varepsilon(\delta,G)\leq m_\varepsilon(\delta,F)$ by Theorem \ref{UBD}, we conclude that the estimation of the topological entropy of free semigroup actions is more accurate in Theorem \ref{LBM}.

\hspace{4mm}
On the other hand, if $m=1$, this means that $G=\{f\}$, then 
	$$\begin{aligned}
	h(f)\geq {\underline{b}}\cdot\limsup_{\delta\to 0}\dfrac{\log (1/\delta)}{\lim_{\varepsilon\to 0 }m_\varepsilon(\delta, f)},
\end{aligned}$$
this yields that Theorem 28 in \cite{RD}.

\end{remark}

\section{The structure of chain transitive systems}\label{chain}

In this section our mainly purpose is to prove Theorem \ref{EC}.
The proof is rather long and technical, so several auxiliary results and definitions are needed.

\hspace{4mm}
Let $(X,d)$ be a compact metric space and $G$ be a free semigroup generated by $m$ generators $f_0,f_1,\cdots,f_{m-1}$  which are continuous maps on $X$.  

\hspace{4mm}
We assume $G$ is chain transitive in this section. For $x\in X$, denote by $T_\varepsilon(x)$  the set of the lengths of all $(w,\varepsilon)$-chain from $x$ to itself with some $w\in F_m^+$ for $G$. Recall that $gcd(\cdot)$ denotes the greatest common divisor.
\begin{lem}
	\label{GCD}
	Let $G$ be chain transitive and $\varepsilon>0$. There exists $k_\varepsilon\geq 1$ such that $gcd(T_\varepsilon(x))=k_\varepsilon$ for any $x\in X$, in the sence that $k_\varepsilon$ does not depend on the choice of $x$.
\end{lem}

\begin{proof}
	This follows from the proof of Lemma 7 in \cite{RD}. For $x\in X$, define $k_\varepsilon:=gcd(T_\varepsilon(x))$. Consider that $y\in X$ and $(y_0=y, y_1,\cdots,y_{n}=y)$ is a $(w,\varepsilon)$-chain from $y$ to itself of length $|w|=n$. We claim that $k_\varepsilon$ divides $n$. Indeed, as $G$ is chain transitive, there are $(w',\varepsilon)$-chain $(x_0=x,x_1,\cdots,x_{m_1}=y)$ from $x$ to $y$ of length $|w'|=m_1$, and $(w'',\varepsilon)$-chain $(z_0=y,z_1,\cdots,z_{m_2}=x)$ from $y$ to $x$ of length $|w''|=m_2$. We have that
	$$
	(x_0=x,x_1,\cdots,x_{m_1}=y,z_1,\cdots,z_{m_2}=x)
	$$
	is a $(w'w'',\varepsilon)$-chain from $x$ to itself of length $m_1+m_2$, and
	$$
	(x_0=x,x_1,\cdots,x_{m_1}=y,y_1,\cdots,y_{n}=y,z_1,\cdots,z_{m_2}=x)
	$$
	is a $(w'ww'',\varepsilon)$-chain from $x$ to itself of length $m_1+n+m_2$. Note that both $m_1+m_2$ and $m_1+n+m_2$ are multiples of $k_\varepsilon$, this yields that $k_\varepsilon$ divides $n$.
\end{proof}

\hspace{4mm}
Define a relation on $X$ by setting $x\thicksim_\varepsilon y$ if there is a $(w,\varepsilon)$-chain from $x$ to $y$ of length a multiple of $k_\varepsilon$ for some $w\in F^+_m$. It is clear that $\thicksim_\varepsilon$ is an equivalence relation. In fact, observe that reflexivity and transitivity are established by Lemma \ref{GCD} and $G$ is chain transitive. On the other hand, we claim that $\thicksim_\varepsilon$ is symmetry, that is,  $y\thicksim_\varepsilon x$ if $x\thicksim_\varepsilon y$. Indeed, consider that  $(y_0=x,\cdots,y_{n}=y)$ is a $(w',\varepsilon)$-chain with $|w|=n$ and $k_\varepsilon\big| n$ from $x$ to $y$. Let we suppose that $(x_0=y,\cdots,x_l=x)$ is a $(w'',\varepsilon)$-chain  with $|w''|=l$ from $y$ to $x$ as $G$ is chain transitive. Then $(y_0=x,\cdots,y_{n}=y, \cdots,x_l=x)$ is a $(w'w'',\varepsilon)$-chain from $x$ to $x$ of length $n+l$. This implies that $l$ is a multiple of $k_\varepsilon$. Consequently, $\thicksim_\varepsilon$ is an equivalence relation.

\begin{remark}
	In fact, by the definition of $k_\varepsilon$, if $x \thicksim_\varepsilon y$, then any $\varepsilon$-chain from $x$ to $y$ must have length a multiple of $k_\varepsilon$. 
\end{remark}

\hspace{4mm}
We now define another equivalence relation on $X$ by saying $x\thicksim y$ if $x\thicksim_\varepsilon y$ for all $\varepsilon>0$.

\begin{lem}
	Let $G$ be chain transitive. For any $\varepsilon>0$, the equivalence relation $\thicksim_\varepsilon$ is both open and closed. The equivalence relation $\thicksim$ is closed.
\end{lem}
\begin{proof}
	Similar to Lemma 9 of \cite{RD}, so we won't repeat it here. 
\end{proof}

\hspace{4mm}
The following result is essentially contained in Exercises 8.22 of \cite{MR1219737}, but we provide a proof here for the reader convenience. 

\hspace{4mm}
Let $T_1,T_2\subset \mathbb{N}$, and $T_1+T_2:=\{t_1+t_2\:|\:t_1\in T_1, t_2\in T_2\}$. $T\subset\mathbb{N}$ is closed under addition if $T+T\subset T$.

\begin{lem} 
	 Let $T\subset \mathbb{N}$ and  $T$ be close under addition, then there exists $N$ so that $nd\in T$ for all $n\geq N$ where $gcd(T)=d$. Furthermore, there exist $r,n\in T$ such that $gcd(r,n)=d$.
	\label{NT}	
\end{lem}
\begin{proof}
	We only prove the case $d=1$, others are similar. We claim that $gcd(t_1,\cdots,t_p)=1$ for some $p\in\mathbb{N}$. Indeed, for any $t_1,t_2\in T$, $gcd(t_1,t_2)=c$. If $c>1$, there is $t_3\in T$ such that $gcd(c,t_3)=c_1$ with $c_1<c$. If $c_1=1$, then $gcd(t_1,t_2,t_3)=1$, otherwise, repeat the above steps. 
	 
	\hspace{4mm}
	To prove the lemma it is to show that there is $N$ such that $n=n_1t_1+\cdots n_pt_p$ for all $n\geq N$ where $n_1,\cdots,n_p\in\mathbb{N}_0$, as $T$ is closed under addition.
	We proceed by induction on $p$. For $p=2$, then $gcd(t_1,t_2)=1$, so $1=n_1t_1-n_2t_2$. If $n\geq N=n_2t_2^2$ then $n=Qt_2+R$ with $Q\geq n_2t_2$ and $0\leq R<t_2$. So $n=(Q-n_2R)t_2+Rn_1t_1.$  Assuming it holds for $p$ we prove it for $p+1$.
	Consider that $gcd(t_1,\cdots,t_{p+1})=1$, and $gcd(t_1,\cdots,t_p)=k$, then $gcd(k,t_{p+1})=1$. Denote 
	$$T_2=\left \{n_1k+n_2t_{p+1}\; \left | \;  n_1,n_2\in\right.\mathbb{N}_0\right \},$$
	and 
	$$
	T_p=\left\{ \left. r_1\frac{t_1}{k}+\cdots+r_p\frac{t_p}{k}\; \right|\;r_1,\cdots,r_p\in\mathbb{N}_0\right \}.
	$$
	Clearly, $T_2$ and $T_p$ are both closed under addition. By hypothesis, let $N_1$, with respect to $T_2$, be the number such that all $n\geq N_1$ implies $n=n_1k+n_2t_{p+1}$ for some $n_1,n_2\in \mathbb{N}_0$, and let $N_2$, with respect to $T_p$, be the number such that all $n\geq N_2$ implies 
	$$
	n=m_1\frac{t_1}{k}+\cdots+m_p\frac{t_p}{k}
	$$ 
	for some $m_1,\cdots,m_p\in\mathbb{N}_0$.
	Take $N=N_1+N_2k$. For all $n\geq N$, $n-N_2k=n_1k+n_2t_{p+1}$, this implies that 
	$$
	\begin{aligned}
	n&=n_1k+n_2t_{p+1}+N_2k\\
	&=(n_1+N_2)k+n_2t_{p+1}\\
	&= (m_1\frac{t_1}{k}+\cdots+m_p\frac{t_p}{k} )k+n_2t_{p+1}\\
	&=m_1t_1+\cdots+m_pt_p+n_2t_{p+1}.
	\end{aligned}$$
This completes our induction.
\end{proof}

\begin{definition}
	\cite{BL}
	Let $J=(j_1,j_2,\cdots)$ be a sequence of integers greater than or equal to $2$. Let $X_J$ be the Cantor set of all sequences $(a_1,a_2,\cdots)$ where $a_i\in\{0,1,\cdots,j_i-1\}$ for all $i$. Define the adding machine map $f_J:X_J\to X_J$ by $$f_J(a_1,a_2,\cdots)=(a_1,a_2,\cdots)+(1,0,0,\cdots),$$
	where addition is defined componentwise$\mod j_i$, with carrying to the right.
\end{definition}

\begin{lem}
	(\cite{AJ}, Th 2.1.3 ) Let $a_1$,$a_2$ be non-negative relatively prime integers. Let
	$$g(a_1,a_2)=a_1a_2-a_1-a_2,$$
	then, for any $N>g(a_1,a_2)$ is representable as a non-negative integer combination of $m$ and $n$, that is, there are $p_1,p_2\in\mathbb{N}_0$ such that $N=p_1a_1+p_2a_2$.
	\label{AJ}
\end{lem}

\begin{proof}[Proof of Theorem \ref{EC}]
	Let  $\varepsilon>0$ and $k_\varepsilon\geq 1$, $k_\varepsilon$ as in Lemma \ref{GCD}. Then $X$ is divided into $k_\varepsilon$ equivalence classes for $\thicksim_\varepsilon$.  In fact, for any $x,z\in X$ with $x\thicksim_\varepsilon z$, let us suppose that
	$$
	(x_0=x,x_1,x_2,\cdots,x_{k_\varepsilon},x_{k_\varepsilon+1},x_{k_\varepsilon+2},\cdots,x_{2k_\varepsilon},\cdots,x_{nk_\varepsilon}=z)
	$$ 
	is a $(w,\varepsilon)$-chain from $x$ to $z$ of length $nk_\varepsilon$. It is clear that
	$$
	x_0\thicksim_\varepsilon x_{k_\varepsilon}\thicksim_\varepsilon x_{2k_\varepsilon}\thicksim_\varepsilon \cdots\thicksim_\varepsilon x_{nk_\varepsilon};
	$$
	$$x_1\thicksim_\varepsilon x_{k_\varepsilon+1}\thicksim_\varepsilon x_{2k_\varepsilon+1}\thicksim_\varepsilon \cdots\thicksim_\varepsilon x_{(n-1)k_\varepsilon+1};$$
	$$\vdots$$
	$$x_{k_\varepsilon-1}\thicksim_\varepsilon x_{2k_\varepsilon-1}\thicksim_\varepsilon x_{3k_\varepsilon-1}\thicksim_\varepsilon \cdots\thicksim_\varepsilon x_{nk_\varepsilon-1}.$$
	
	Denote these equivalence classes as $[x_0]_{\thicksim_{\varepsilon}},[x_1]_{\thicksim_{\varepsilon}},\cdots,[x_{k_\varepsilon-1}]_{\thicksim_{\varepsilon}}$, respectively. For any $x\in X$, we deduce that $x\in [x_i]_{\thicksim_{\varepsilon}}$ for some $0\leq i \leq k_\varepsilon-1$ by $G$ is chain transitive. And we have $x_i\nsim_\varepsilon x_{i+1}$ for each $i=0,\cdots,k_\varepsilon-1$ since $(x_i,x_{i+1})$ is an $\varepsilon$-chain of length $1$.  Consequently, $X$ is divided into $k_\varepsilon$ equivalence classes. 
	
	\hspace{4mm}
	Obviously, $G$ cyclically permutes $k_\varepsilon$ equivalence classes, that is, 
	$$
	[f_j(x_i)]_{\thicksim_{\varepsilon}}=[x_{i+1\!\!\!\mod k_\varepsilon}]_{\thicksim_{\varepsilon}}
	$$ 
	for all $j=0,\cdots,m-1$ and $i=0,\cdots,k_\varepsilon-1$. Moreover, every equivalence class is invariant under $G^{k_\varepsilon}$, that is, $f_w([x_i]_{\thicksim_{\varepsilon}})\subset [x_i]_{\thicksim_{\varepsilon}}$ for all $i=0,\cdots,k_\varepsilon-1$ and $w\in F^+_m$ with $|w|=k_\varepsilon$. Since $[f_w(x_i)]_{\thicksim_{\varepsilon}}= [x_i]_{\thicksim_{\varepsilon}}$, we only prove $ f_w([x_i]_{\thicksim_{\varepsilon}})\subset [f_w(x_i)]_{\thicksim_{\varepsilon}}$. For any $x\in [x_i]_{\thicksim_{\varepsilon}}$, we have $f_w(x) \thicksim_{\varepsilon} f_w(x_i)$ since $x\thicksim_{\varepsilon} f_w(x)$, $x \thicksim_{\varepsilon} x_i$ and $x_i\thicksim_{\varepsilon} f_w(x_i)$.  
	
	\hspace{4mm}
	The quantity $k_\varepsilon$ is nondecreasing as $\varepsilon\to 0$, and in fact $k_{\varepsilon_2}$ divides $k_{\varepsilon_1}$ if $\varepsilon_1\leq\varepsilon_2$, since an $\varepsilon_1$-chain is an $\varepsilon_2$-chain. Either $k_\varepsilon$ stabilizes at some $k$, or it grows without bound. We consider the three cases separately.

    \hspace{4mm}
    Case 1: $k_\varepsilon$ stabilizes at $k=1$. Then there is only one $\thicksim$ equivalence class, and $G$ is chain mixing. Indeed, for any $\varepsilon>0$ and any $x\in X$, $T_\varepsilon(x)$ is closed under addition. Since $gcd(T_\varepsilon(x))=1$, there exist $r_1,s_1\in T_\varepsilon(x)$ such that $gcd(r_1,s_1)=1$ by Lemma \ref{NT}. This implies that there exist $(w',\varepsilon)$-chain and $(w'',\varepsilon)$-chain from $x$ to itself of lengths $|w'|=r_1$ and $|w''|=s_1$, respactively. We can get an $\varepsilon$-chain from $x$ to itself of length $N$ for any $N>g(r_1,s_1)$ by Lemma \ref{AJ}. By compactness, there is a $p\in\mathbb{N}$ such that $\bigcup_{i=1}^pB(x_i,\frac{\varepsilon}{2})=X$, let the length of an $\frac{\varepsilon}{2}$-chain from $x_i$ to $x_j$ be $M_{ij}$, and $M=\max_{1\leq i,j\leq p}M_{ij}$, then for between any two points in $X$ there is an $\varepsilon$-chain of length less than or equal to $M$. Hence, for any $y\in X$, any $n>g(r_1,s_1)+M$, there is a $(w,\varepsilon)$-chain of length $|w|=n$ from $x$ to $y$. Therefore, $G$ is chain mixing.

    \hspace{4mm}
    Case 2: $k=k_{\varepsilon}$ for sufficiently small $\varepsilon$. Then the equivalence relation $\thicksim$ is the same as $\thicksim_{\varepsilon}$. Thus there are $k$ equivalence classes, $G^k$ cycles among the classes periodically, and each class is invariant under $G^k$. An argument similar to that for Case 1 show that $([x_i]_{\thicksim},G^k)$ is chain mixing for each $i=0,\cdots,k-1$. By uniform continuity, pick $\varepsilon'<\varepsilon/k$ small enough that for any $(w,\varepsilon')$-chain of length $|w|=k$ for $G$, denotes $(x_0,\cdots,x_n)$, we have $d(f_{\overline{w}} (x_0),x_k)<\varepsilon$. Notice that $k_{\varepsilon'}=k$ as $\varepsilon'<\varepsilon$, then  $gcd(T_{\varepsilon'}(x))=k$ for any $x\in [x_i]_{\thicksim}$. By Lemma \ref{NT}, there exist $r_2,s_2\in T_{\varepsilon'}(x)$ such that $gcd(r_2,s_2)=k$. This implies that there are $(w',\varepsilon')$-chain and $(w'',\varepsilon')$-chain from $x$ to itself of lengths $|w'|=r_2$ and $|w''|=s_2$ for $G$, respactively. Put $r_2=ak$ and $s_2=bk$, then $gcd(a,b)=1$.  By Lemma \ref {AJ}, for any $N$ with $N>g(a,b):=ab-a-b$, there are $p,q\in\mathbb{N}_0$ such that $N=pa+qb$ , hence there is an $\varepsilon'$-chain from $x$ to itself of length $kN$. Since $[x_i]_{\thicksim}$ is closed, there is $p'\geq 1$ suth that $\bigcup_{j=1}^{p'}B(z_j,\frac{\varepsilon'}{2})\supset [x_i]_{\thicksim}$ where $z_j\in [x_i]_{\thicksim}$ for all $j=1,\cdots,p'$. For $1\leq j,r\leq p'$, there is a $(w_{jr},\varepsilon')$-chain from $z_j$ to $z_r$ of length $|w_{jr}|=kM_{j,r}$ for $G$. Let $M=\max_{1\leq j,r\leq p'}M_{j,r}$, then between any two points in $[x_i]_{\thicksim}$ there is an $\varepsilon'$-chain for $G$ of length equal to $ck$ for some $1\leq c\leq M$. We chaim that for any $n>g(a,b)+M$ and any $y\in [x_i]_{\thicksim}$, there is an $\varepsilon$-chain from $x$ to $y$ of the length $n$ for $G^k$. Indeed, notice that there is a $(w''',\varepsilon')$-chain from $x$ to $y$ of length $ck$ with $1\leq c\leq M$ for $G$. Since $n-c>g(a,b)$, there are $p_1,q_1\in\mathbb{N}_0$ such that $n-c=p_1a+q_1b$ by Lemma \ref{AJ}. Then from the above structure, a $(w,\varepsilon')$-chain for $G$ of length $nk$, $(x_0=x,\cdots,x_{nk}=y)$, is naturally formed, where
    $$
    \begin{aligned}
	    w=&\underbrace{w'\cdots w'}_{p_1}\underbrace{w''\cdots w''}_{q_1}w'''\\
	    :=&i_0\cdots i_{k-1}i_k\cdots i_{2k-1}i_{2k}\cdots i_{nk-1}.
    \end{aligned}
    $$
    Then, we have 
    $$
    d\left (f_{\overline{i_{jk}\cdots i_{{(j+1)}k-1}}}(x_{jk}),x_{{(j+1)}k}\right )<\varepsilon
    $$
    for all $j=0,\cdots,n-1$. Moerover, there is $u=u_0\cdots u_{n-1}\in F^+_{m^k}$ such that $(f)_{u_j}=f_{\overline{i_{jk}\cdots i_{{(j+1)}k-1}}}$ for all $j=0,\cdots ,n-1$, this means that  $(x_0,x_k,x_{2k},\cdots,x_{nk})$ is a $(u,\varepsilon)$-chain of length $|u|=n$ for $G^k$. 
    Consequently, $([x_i]_{\thicksim},G^k)$ is chain mixing. 
            
    \hspace{4mm}
    Case 3: $k_\varepsilon$ grows without bound as $\varepsilon$ decreasing to $0$. Then the period of $G$'s cycling goes to infinity as $\varepsilon$ shrinks to $0$. Let $\widetilde{K}_\varepsilon=X/\thicksim_{\varepsilon}=\{[x_0]_{\thicksim_{\varepsilon}},\cdots,[x_{k_\varepsilon-1}]_{\thicksim_{\varepsilon}}\}$ and $\widetilde{K}=X/\thicksim$ be the quotient spaces with the quotient topology. For $j=0,\cdots,m-1$, we define the induced map on $\widetilde{K}_\varepsilon$  to be given by  $$\widetilde{f}_{j}:\widetilde{K}_\varepsilon\to\widetilde{K}_\varepsilon,\; [x_i]_{\thicksim_{\varepsilon}}\mapsto [f_j(x_i)]_{\thicksim_{\varepsilon}}$$ for all $i=0,\cdots,k_\varepsilon-1$. Since 
    $
    \widetilde{f}_{j}([x_i]_{\thicksim_{\varepsilon}})= [f_j(x_{i})]_{\thicksim_{\varepsilon}}=[x_{i+1\!\!\!\mod k_\varepsilon}]_{\thicksim_{\varepsilon}}
    $ 
    for all $j=0,\cdots,m-1$ and $i=0,\cdots,k_\varepsilon-1$, then every induced map $\widetilde{f}_{j}$ is the same, denoted as $\widetilde{f}_\varepsilon:\widetilde{K}_\varepsilon\to\widetilde{K}_\varepsilon$, $[x_i]_{\thicksim_{\varepsilon}}\mapsto[x_{i+1\!\!\!\mod k_\varepsilon}]_{\thicksim_{\varepsilon}}$ for all $i=0,\cdots,k_\varepsilon-1$. Similar to Theorem 6 of \cite{RD}, we deduce that $(\widetilde{K},\widetilde{f})$ is topological conjugate to an adding machine map $f_J$. Therefore, $G$ factors onto an adding machine map $f_J$.  
\end{proof}

\begin{cor}\label{5.7}
	Let $X$ be connected and $G=\{f_0,\cdots,f_{m-1}\}$ where $f_{i}:X\to X$ is continuous for all $0\leq i\leq m-1$. Then the following are euqivalent:
	\begin{enumerate}[(1)]
		\item $G$ is chain recurrence;
		\item $G$ is chain transitive;
		\item $G$ is totally chain transitive;
		\item $G$ is chain mixing.
	\end{enumerate}
\end{cor}
\begin{proof}
	It follows from the proof of Corollary 14 of \cite{RD}. Clearly, $(4)\Rightarrow (3)\Rightarrow(2)\Rightarrow(1)$. By Theorem \ref{EC}, $X$ is connected and $G$ is chain transitive imply that $G$ is chain mixing. So it enough to show that chain recurrence implies chain tansitivity. Assume that $G$ is chain recurrence and $\varepsilon>0$. We say that $x$ and $y$ are $\varepsilon$-chain equivalent if there are $(w',\varepsilon)$-chain from $x$ to $y$ and $(w'',\varepsilon)$-chain from $y$ to $x$ for some $w',w''\in F_m^+$. This is an equivalent relation as $G$ is chain recurrence. By connectivity of $X$, so it suffices to show that this is an open equivalent relation. Consider that $x$ and $y$ are $\varepsilon$-chain equivalent. Choose $\delta\leq \varepsilon/2$ such that $d(y,y')<\delta$ implies $d(f_i(y),f_i(y'))<\varepsilon/2$ for all $i=0,\cdots,m-1$. It suffices to show that $x$ is $\varepsilon$-chain equivalent to an arbitrary $y'\in B(y,\delta)$. Let $(x_0=x,\cdots,x_n=y)$ be a $(w',\varepsilon)$-chain for some $w'\in F_m^+$ from $x$ to $y$, and $(y_0=y,\cdots,y_m=y)$ be a $(w'',\varepsilon/2)$-chain for some $w''\in F_m^+$ from $y$ to itself. Then $$(x_0=x,\cdots,x_n=y,y_1,\cdots,y_{m-1},y')$$ is a $(w'w'',\varepsilon)$-chain from $x$ to $y'$. Similarly, let $(z_0=y,\cdots,z_r=x)$ be a $(w''',\varepsilon)$-chain for some $w'''\in F_m^+$. Then $$(y',y_1,\cdots,y_m=y,z_1,\cdots,z_r=x)$$ is a $(w''w''',\varepsilon)$-chain from $y'$ to $x$. This yields that $x$ and $y'$ are $\varepsilon$-chain equivalent.
\end{proof}
\begin{remark}
	If $X$ be connected and $G=\{f\}$, by Corollary \ref{5.7}, this gives a generalization of Corollary 14 of \cite{RD}. 
\end{remark}

\textbf{Acknowledgements}
The  work was supported by National Natural Science Foundation of China (grant no.11771149).

\end{document}